\documentclass[11pt,reqno]{amsart}
\usepackage{amsmath,amsthm,amscd,amsfonts,amssymb,color}
\usepackage{cite}
\usepackage{graphicx}
\usepackage[mathscr]{eucal}
\usepackage[bookmarksnumbered,colorlinks,plainpages]{hyperref}
\newtheorem{theorem}{Theorem}[section]
\newtheorem{lemma}[theorem]{Lemma}
\newtheorem{Acknowledgement}{Acknowledgement}

\newtheorem{corollary}[theorem]{Corollary}
\theoremstyle{definition}
\newtheorem{definition}[theorem]{Definition}

\theoremstyle{remark}
\newtheorem{remark}[theorem]{Remark}
\numberwithin{equation}{section}
\def\DJ{\leavevmode\setbox0=\hbox{D}\kern0pt\rlap
 {\kern.04em\raise.188\ht0\hbox{-}}D}

\begin{document}

\title[Hyre-Ulam Stability Results]{Generalized Baker's Result and Stability of Functional Equations using fixed point results }
\author[S. \ Laha, L.K. \ Dey]
{Supriti Laha$^{1}$,  Lakshmi Kanta Dey$^{2}$}

\address{{$^{1}$\,} Supriti Laha,
                    Department of Mathematics,
                    National Institute of Technology
                    Durgapur, India.}
                    \email{lahasupriti@gmail.com}
\address{{$^{2}$\,}  Lakshmi Kanta Dey,
                    Department of Mathematics,
                    National Institute of Technology
                    Durgapur, India.}
                    \email{lakshmikdey@yahoo.co.in}


\begin{abstract}
Hyre-Ulam stability of functional equation in single variable is studied in non-triangular metric spaces. We derive it as applications of some fixed point results developed on the said structure. A general version of Baker's theorem is also deduced as a consequence.
\end{abstract}
 
\maketitle

\setcounter{page}{1}
%
%

{\Large{\section{ {\bf{Introduction}}}}}

The study of stability of functional equation was started by S. M. Ulam \cite{brzdek} in 1940, on attending the following question: Consider $G$ and $(G_m, \rho)$ as a group and a metric group respectively. Given $\varepsilon> 0$, is it possible to find a $\delta> 0$ such that if $\mathcal{F}:G\rightarrow G_m$ satisfies
$$\rho(\mathcal{F}(g_1\circ g_2),\mathcal{F}(g_1)\star \mathcal{F}(g_2))<\delta, \quad \quad g_1,g_2\in G,$$ then there exists a homomorphism $\mathcal{H}:G\rightarrow G_m$  which approximates $\mathcal{F},$ i.e.,  $$\rho(\mathcal{F}(g),\mathcal{H}(g))\leq \varepsilon, \quad \quad g\in G~?$$ Alternatively, one can pose the stability problem as: is it possible to approximate an $\varepsilon$-homomorphism (approximate homomorphism) by a homomorphism? In 1941, D. H. Hyers gave this result: 

\begin{theorem}\cite{H1}
Let $B_1$ and $B_2$ be Banach spaces and let $h:B_1\rightarrow B_2$ be a transformation such that, for some $\delta> 0$,
$$\lVert{h(u+v)-h(u)-h(v))\rVert}<\delta, \quad \quad u,v\in B_1.$$
Then the limit
$$g(u)=\lim_{m\rightarrow \infty}\frac{h(2^m u)}{2^m},$$
exists for each $u\in B_1$ and $g:B_1\rightarrow B_2$ is the unique additive transformation satisfying
$$\lVert{h(u)-g(u)\rVert}<\delta, \quad \quad u\in B_1.$$
\end{theorem}
It provides a partial answer of the above question.
T. Aoki, D. G. Bourgin, Rassias and other researchers investigated it in the context of linear transformations in Banach sapaces (see \cite{A1}, \cite{B1}, \cite{R1}, \cite{R3}). These studies developed the theory of stability of functional equations. Afterwards, the stability theory has been   comprehensively investigated in \cite{C1}, \cite{J1}, \cite{J2}, \cite{R2} for the functional equations: 
\begin{align*}
h(s+t) & =h(s)+h(t)~  &(\text{additive}),\\
h\left(\frac{s+t}{2}\right) & =\frac{h(s)+g(t)}{2}~ &(\text{Jensen's}),\\
h(s+t)+h(s-t) & =2h(s)+2h(t)~ &(\text{quadratic}).
\end{align*}
The stability problem has also been studied for isometries and continuous functions in \cite{B2}, \cite{B3}, \cite{HU1}, \cite{HU2}. In 1978, P. M. Gruber\cite{G1} established some fascinating results on stability of isometries.  The stability theory for various functional equations explored extensively in multi variables (see \cite{B2}, \cite{B1}, \cite{C1}, \cite{HU1}, \cite{HU2}, \cite{J1}, \cite{J2}).

\medskip

In case of single variable, functional equations are of the form 
\begin{equation}\label{eq0}
g(s)=G(s,g(\psi(s)))
\end{equation} 
where  $g:S\rightarrow X$ and $G: S\times X\rightarrow X$ are mappings.
J. A. Baker\cite{JB1}, the first mathematician who examined the stability of equation (\ref{eq0}) and established the following result:
 
\begin{theorem}\cite{JB1}\label{TB1}
Let $S$ be a non-empty set and $(X,\rho)$ be a complete metric space. Consider $\psi:S\rightarrow S,$ $G: S\times X\rightarrow X$ and $0\leq \lambda <1.$ If \begin{enumerate}
\item $\displaystyle\sup_{s\in S} {\rho(g(s),G(s,g(\psi(s)))}\leq \delta$,   for some $g\in X^S$ and some $\delta>0,$
\item $\displaystyle\sup_{s\in S} {\rho(G(s,u_1),G(s,u_2))}\leq \lambda \rho(u_1,u_2),$    $\forall ~u_1,u_2\in X$ and $\lambda \in [0,1).$
\end{enumerate}

Then there exists a unique function $f:S\rightarrow X$ that satisfy equation (\ref{eq0}) and 
$$\sup_{s\in S} \{\rho(f(s),g(s))\}\leq \frac{\delta}{1-\lambda}.$$
\end{theorem}
It is evident that Theorem \ref{TB1} is an application of  Banach contraction principle.

Motivated by Baker's result described in Theorem \ref{TB1}, this article aims to explore the stability of equation (\ref{eq0}) by employing various fixed point results (Theorems \ref{tub}, \ref{tuk}, \ref{tuc} and \ref{tug}) in non-triangular metric spaces.

We divide this article in five sections. Second section contains necessary definitions and useful results. In section 3, we develop well-known Banach, Kannan, Chatterjea and \'{C}iri\'{c} fixed point theorems in a complete non-triangluar metric space. In section 4, we establish the stability of equation (\ref{eq0}) in Hyre-Ulam sense as applications of the results derived in section 3. Finally, we present a general variant of Baker's theorem \cite{JB1} as a consequence of Corollary \ref{tubur}.

{\Large{\section{{\textbf {Preliminaries}}}}}

In the year 2020, F. Khojasteh and H. Khandani \cite{KK1} introduced the notion of \textit{non-triangular metric space}. This topological space generalize \textit{JS-metric space} introduced by M. Jleli and B. Samet\cite{JS1}, in most cases. Therefore, \textit{non-triangular metric space} contains a larger collection of topological spaces those are \textit{b-metric spaces, modular spaces} and so on. We now recall the definition of \textit{non-triangular metric space}.

\begin{definition}\cite{KK1}
Let $X$ be a non-empty set.  Then for a non-negative function $\rho$ on $X\times X$, $(X,\rho)$ is called a \textit{non-triangular metric space} if the following hold: 
\begin{enumerate}
\item[(N1)] $\rho(u,u)=0,~ \forall u\in X.$
\item[(N2)] $\rho(u,v)=\rho(v,u), ~ \forall u,v\in X.$
\item[(N3)]For every $u,u'\in X$ and $\{u_m\} \in X$, $\displaystyle\lim_{m\rightarrow \infty} \rho(u_m,u)=0$ and $\displaystyle\lim_{m\rightarrow \infty} \rho(u_m,u')=0 \implies u=u'.$ 
\end{enumerate}
\end{definition}

\begin{remark}
$\rho(u,v)=0 \implies u=v$.
\end{remark}

%
\noindent
Definitions of convergence of a sequence, Cauchy sequence, and completeness in a non-triangular metric space are given in [Definition 6,\cite{KK1}].

\begin{remark}\label{ul}
Limit of a $\rho$-convergent sequence in a non-triangular metric space is unique.

\end{remark}

\begin{definition}\label{oc}
Let $X$ be a \textit{non-triangular metric space} and $\mathcal{R}:X\rightarrow X$ be a mapping. Then the set, $O(u,\infty)=\{u, \mathcal{R}(u), \mathcal{R}^2(u),\cdots\}$ is said to be an \textit{orbit} of $u$. $\mathcal{R}$ is called \textit{orbitally continuous} if a sequence $\{u_m\}\in O(u,\infty)$ that converges to $u^*$, then $\mathcal{R}(u_m)$ converges to $\mathcal{R}(u^*)$.
\end{definition}

\noindent
Some well-known contraction mappings such as Banach, Kannan, and Chatterjea contraction mapping from a metric space to  a non-triangular metric space are appended below.

\begin{definition}
Let $(X, \rho)$ be a \textit{non-triangular metric space}. Then a mapping $\mathcal{R}$ on $X$ is defined as:
\begin{enumerate}
\item a \textit{$\lambda_{B}$-contraction} if $\rho(\mathcal{R}(u),\mathcal{R}(v))\leq \lambda \rho(u,v), \quad \quad \lambda\in [0,1);$
\item a \textit{$\lambda_{K}$-contraction}  if $\rho(\mathcal{R}(u),\mathcal{R}(v))\leq \lambda [\rho(\mathcal{R}(u),u)+\rho(\mathcal{R}(v),v)], \quad \lambda\in [0,\frac{1}{2});$
\item a \textit{$\lambda_{C}$-contraction}  if $\rho(\mathcal{R}(u),\mathcal{R}(v))\leq \lambda [\rho(\mathcal{R}(u),v)+\rho(\mathcal{R}(v),u)], \quad \lambda\in [0,\frac{1}{2});$
\end{enumerate}
for all $u,v\in X.$
\end{definition}

\noindent

{\Large{\section{{\textbf{Fixed Point Theorems}}}}}\label{S3}

Khojasteh and Khandani\cite{KK1} illuminated that a number of fixed point theorems can be proved without engaging the \textit{triangle inequality} property. The authors established the Banach contraction theorem in non-triangular metric spaces in [Corollary 3.1.1, \cite{KK1}].
\begin{theorem} [cf.\cite{KK1}]\label{tub}
Let $(X,\rho)$ be a complete non-triangular metric space. Suppose $\mathcal{R}$ be a $\lambda_B$-contraction on $X$ such that $\displaystyle{\sup_{m\in \mathbb{N}}}\{\rho(u_0,\mathcal{R}^m(u_0))\}\leq \delta$ for some $u_0\in X$ and some $\delta >0.$ Then a unique point $u_0^* \in X$ exists with $\mathcal{R}(u_0^*)= u_0^*$.
\end{theorem}

\noindent
We prove a variant of the above theorem by weakening a condition and taking triangle inequality only on an orbit.

\begin{theorem}\label{tub1}
Let $(X,\rho)$ be a complete non-triangular metric space. Suppose $\mathcal{R}$ be a $\lambda_B$-contraction on $X$ such that $\rho(u_0,\mathcal{R}(u_0))\leq \delta$ for some $u_0\in X$ and some $\delta >0.$ Also, triangle inequality holds on an orbit of $u_0\in X.$ Then we get a unique point $u_0^*\in X$ satisfying $\mathcal{R}(u_0^*)= u_0^*$. Moreover, $\rho(u_0,u_0^*)\leq \dfrac{\delta}{1-\lambda}$.
\end{theorem}
\begin{proof}
Evidently, triangle inequality on an orbit $u_0$ combined with $\rho(u_0,\mathcal{R}(u_0))\leq \delta$ gives us $\displaystyle{\sup_{m\in \mathbb{N}}}\{\rho(u_0,\mathcal{R}^m(u_0))\}\leq \delta$. Therefore, Theorem \ref{tub} ensures the existence of a unique fixed point of $\mathcal{R}$.
Now using triangle inequality on an orbit, we get
\begin{align*}
\rho(u_0,T^{n}(u_0)) &\leq \rho(u_0,\mathcal{R}(u_0))+\lambda \rho(u_0,\mathcal{R}(u_0))+\dots+\lambda^{n-1}  \rho(u_0,\mathcal{R}(u_0))\\
\implies \rho(u_0,u_0^*)& = \sup_{n\in \mathbb{N}}\rho(u_0,T^{n}(u_0)) \leq \dfrac{\delta}{1-\lambda}.
\end{align*}
\end{proof}

\noindent
Karapinar et al. established in [Corollary 2.9, \cite{EK}], that a $\lambda_K$-contraction mapping (Kannan type) has a fixed point if the following holds:
\begin{align*}
\delta_1(\mathcal{R},u_0)=\sup\{\mathcal{R}^m(u_0),\mathcal{R}^k(u_0):m,k\geq 1\}<\infty.
\end{align*} 
\noindent
Here we will use $\rho(u_0,\mathcal{R}(u_0))\leq \infty,$ to deduce the result.

\begin{theorem}\label{tuk}
Let $(X,\rho)$ be a complete \textit{non-triangular metric space}. Consider a $\lambda_{K}$-contraction $\mathcal{R}$ on $X$. Let for some $u_0\in X$ and $\delta>0$, $\rho(u_0,\mathcal{R}(u_0))\leq \delta$. If $\mathcal{R}$ is orbitally continuous, then $\mathcal{R}$ has a unique fixed point $u_0^*$. 
\end{theorem}

\begin{proof}
By the hypothesis, there is a $u_0 \in X$, $$ \rho(u_0,\mathcal{R}(u_0))\leq \delta.$$ Since $\mathcal{R}$ is $\lambda_{K}$-contraction; for all $j\in \mathbb{N}$, we obtain \begin{equation}\label{eq:4}
 \rho(\mathcal{R}^m(u_0),\mathcal{R}^{m+j}(u_0))\leq \lambda [\rho(\mathcal{R}^{m-1}(u_0),\mathcal{R}^m(u_0))+ \rho(\mathcal{R}^{m-1+j}(u_0),\mathcal{R}^{m+j}(u_0))].
\end{equation} 
Again $$\rho(\mathcal{R}^{m-1}(u_0),\mathcal{R}^{m}(u_0))\leq \lambda [\rho(\mathcal{R}^{m}(u_0),\mathcal{R}^{m-1}(u_0)))+ \rho(\mathcal{R}^{m-1}(u_0),\mathcal{R}^{m-2}(u_0))].$$
It implies that 
\begin{align*}
\rho(\mathcal{R}^{m-1}(u_0),\mathcal{R}^{m}(u_0))
& \leq \frac{\lambda}{1-\lambda} \rho(\mathcal{R}^{m-1}(u_0),\mathcal{R}^{m-2}(u_0)))\\
& \vdots \\
& \leq \frac{\lambda^{m-1}}{(1-\lambda)^{m-1}} \rho(\mathcal{R}(u_0),u_0)).
\end{align*}
Therefore, from equation (\ref{eq:4}) we get $$\rho(\mathcal{R}^{m}(u_0),\mathcal{R}^{m+j}(u_0))\leq  \frac{\lambda^{m}}{(1-\lambda)^{m-1}} \rho(\mathcal{R}(u_0),u_0)+ \frac{\lambda^{m+j}}{(1-\lambda)^{m+j-1}}\rho(\mathcal{R}(u_0),u_0).$$ It implies
$$\lim_{m\rightarrow \infty} \{\rho(\mathcal{R}^{m}(u_0),\mathcal{R}^{m+j}(u_0))\}=0, ~\forall j\in \mathbb{N}. $$
Therefore, $$\lim_{m\rightarrow \infty}\sup \{\rho(\mathcal{R}^{m}(u_0),\mathcal{R}^{m+j}(u_0)):j\geq 1\}=0. $$ Thus, $\{\mathcal{R}^m(u_0)\}$ is $\rho$-Cauchy. Using $\rho$-completeness of $X$, we get a point $u_0^*\in X$ satisfying $$\lim_{m\rightarrow \infty}\rho(\mathcal{R}^{m}(u_0),u_0^*)=0.$$ Since, $\mathcal{R}$ is orbitally continuous, $\{\mathcal{R}(\mathcal{R}^m(u_0))\}$ converges to $\mathcal{R}(u_0^*)$. Using Remark \ref{ul}, we get $$\mathcal{R}(u_0^*)=u_0^*.$$

\noindent
For uniqueness, let $u_0^*$ and $v_0^*$ be two fixed points of $\mathcal{R}.$ Then,
\begin{align*}
& \rho(\mathcal{R}(u_0^*),\mathcal{R}(v_0^*))  \leq \lambda[\rho(\mathcal{R}(u_0^*),u_0^*)+\rho(\mathcal{R}(v_0^*),v_0^*)],\\
&  \implies \rho(u_0^*,v_0^*)  \leq 0.
\end{align*}
Therefore, $u_0^*=v_0^*.$
\end{proof}

\begin{remark}
Let the triangle inequality hold on an orbit of $u_0$, for some $u_0\in X$, then $$\rho(u_0,u_0^*)\leq \dfrac{(1+\lambda)\delta}{1-2\lambda}.$$
\end{remark}
 

\noindent
Next we enquire whether a $\lambda_C$-contraction mapping has a fixed point or not. The following result reflects that the answer is affirmative.
\begin{theorem}\label{tuc}
Let $(X,\rho)$ be a complete \textit{non-triangular metric space}. Consider a $\lambda_{C}$-contraction $\mathcal{R}$ on $X$. Suppose $u_0\in X$, $\delta>0$ such that $\rho(u_0,\mathcal{R}(u_0))\leq \delta$ and triangle inequality holds on an orbit of $u_0$. If  $\mathcal{R}$ is orbitally continuous, then it has a unique fixed point $u_0^*$ with $$\rho(u_0,u_0^*)\leq \dfrac{(1+\lambda)\delta}{1-2\lambda}.$$
\end{theorem}

\begin{proof}
From the hypothesis, we have, $$\rho(u_0,\mathcal{R}(u_0))\leq \delta, \text{ for some }u_0 \in X.$$ Since $\mathcal{R}$ is a $\lambda_{C}$-contraction,  $$\rho(\mathcal{R}^{m-1}(u_0),\mathcal{R}^{m}(u_0))\leq \lambda \rho(\mathcal{R}^{m}(u_0),\mathcal{R}^{m-2}(u_0))).$$ By using triangle inequality on $O(u_0,\infty)$, we get 
\begin{align*}
\rho(\mathcal{R}^{m-1}(u_0),\mathcal{R}^{m}(u_0))
& \leq \frac{\lambda}{1-\lambda} \rho(\mathcal{R}^{m-1}(u_0),\mathcal{R}^{m-2}(u_0)))\\
& \vdots \\
& \leq \frac{\lambda^{m-1}}{(1-\lambda)^{m-1}} \rho(\mathcal{R}(u_0),u_0)).
\end{align*}
Let us denote $\gamma=\frac{\lambda}{1-\lambda} $, then $\gamma<1$.
Using triangle inequality on an orbit of $u_0$, we get
 \begin{align*}
\rho(\mathcal{R}^{m}(u_0),\mathcal{R}^{m+j}(u_0))
& \leq \rho(\mathcal{R}^{m}(u_0),\mathcal{R}^{m+1}(u_0))+\cdots +\rho(\mathcal{R}^{m+j-1}(u_0),\mathcal{R}^{m+j}(u_0))\\
& \leq \gamma^m \rho(\mathcal{R}(u_0),u_0)+\gamma^{m+1}\rho(\mathcal{R}(u_0),u_0)+\cdots +\gamma^{m+j-1}\rho(\mathcal{R}(u_0),u_0)\\
& \leq \gamma^m \rho(\mathcal{R}(u_0),u_0)[1+\gamma+\gamma^2+\cdots+ \gamma^{j-1}].
\end{align*}
This implies  $$\lim_{m\rightarrow \infty} \{\rho(\mathcal{R}^{m}(u_0),\mathcal{R}^{m+j}(u_0))\}=0, ~\forall j\in \mathbb{N}. $$
Thus,
$$\lim_{m\rightarrow \infty}\sup \{\rho(\mathcal{R}^{m}(u_0),\mathcal{R}^{m+j}(u_0)):j\geq 1\}=0. $$ Therefore, $\{\mathcal{R}^m(u_0)\}$ is $\rho$-Cauchy. Then, $\rho$-completeness of $X$ provides a point $u_0^*\in X$ that satisfies $$\lim_{m\rightarrow \infty}\rho(\mathcal{R}^{m}(u_0),u_0^*)=0.$$ By using orbital continuity of $\mathcal{R}$ and Remark \ref{ul}, we obtain $\mathcal{R}(u_0^*)=u_0^*.$ Uniqueness can be checked easily.

\noindent

\noindent
Since the triangle inequality holds on an orbit of $u_0$, we obtain \[
\rho(u_0,u_0^*) \leq \frac{(1+\lambda)\delta}{1-2\lambda}.
\]
\end{proof}

\noindent
Finally, we consider \'Ciri\'c's\cite{C2} generalized contraction mapping. It is a generalization of all aforementioned contraction mappings as well as a larger class of mappings that are not mentioned in this study. Next, we obtain a fixed point result for [Theorem 2.5, \cite{C2}] in a non-triangular metric space.

\begin{theorem}\label{tug}
Let $(X,\rho)$ be a complete \textit{non-triangular metric space} and $\mathcal{R}:X\rightarrow X$ be a mapping such that
\begin{align*}
\rho(\mathcal{R}(u_1),\mathcal{R}(u_2))
&\leq \lambda_1\rho(u_1,u_2)+\lambda_2\rho(\mathcal{R}(u_1),u_1) +\lambda_3\rho(\mathcal{R}(u_2),u_2)\\
&+\lambda_4\rho(\mathcal{R}(u_1),u_2)+\lambda_5\rho(\mathcal{R}(u_2),u_1),
\end{align*}
where $\lambda_i:X\times X\rightarrow \mathbb{R}_+$, $i\in\{1, \dots, 5\}$ and $\sum_{i=1}^{5}\lambda_{i}\leq \lambda, ~ \forall u_1,u_2\in X$ and $\lambda\in [0,1).$ Let $\rho(u_0,\mathcal{R}(u_0))\leq \delta$, and triangle inequality holds on an orbit of $u_0$, for some $u_0\in X$, $\delta>0$. If $\mathcal{R}$ is orbitally continuous, then a unique point $u_0^*$ exists with $\mathcal{R}(u_0^*)=u_0^*$. Also, $$\rho(u_0,u_0^*)\leq \dfrac{(2+\lambda)\delta}{2(1-\lambda)}.$$
\end{theorem}

\begin{proof}
By the given hypothesis we have, 
\begin{align} 
\rho(\mathcal{R}(u_1),\mathcal{R}(u_2))
&\leq \lambda_1\rho(u_1,u_2)+\lambda_2\rho(\mathcal{R}(u_1),u_1) +\lambda_3\rho(\mathcal{R}(u_2),u_2)\notag \\
&+\lambda_4\rho (\mathcal{R}(u_1),u_2)+\lambda_5\rho(\mathcal{R}(u_2),u_1),\label{eq:1}
\end{align}
for each $u_1,u_2\in X.$ 
Interchanging the role of $u_1$ and $u_2$, we get
\begin{align} 
\rho(\mathcal{R}(u_1),\mathcal{R}(u_2))
&\leq \lambda_1\rho(u_2,u_1)+\lambda_2\rho(\mathcal{R}(u_2),u_2) +\lambda_3\rho(\mathcal{R}(u_1),u_1)\notag \\
&+\lambda_4\rho(\mathcal{R}(u_2),u_1)+\lambda_5\rho(\mathcal{R}(u_1),u_2), \label{eq:2}
\end{align}
for each $u_1,u_2\in X.$
Adding (\ref{eq:1}) and (\ref{eq:2}), then dividing by $2$, we get
\begin{align} 
\rho(\mathcal{R}(u_1),\mathcal{R}(u_2))
&\leq \lambda_1\rho(u_1,u_2)\notag \\
&+\frac{\lambda_2+\lambda_3}{2}[\rho(\mathcal{R}(u_1),u_1) +\rho(\mathcal{R}(u_2),u_2)]\notag \\
&+\frac{\lambda_4+\lambda_5}{2}[\rho(\mathcal{R}(u_1),u_2) +\rho(\mathcal{R}(u_2),u_1)],
\end{align}
for each $u_1,u_2\in X$.
Assume that $u_0\in X$ and $m\in \mathbb{N}$. Then
\begin{align*} 
\rho(\mathcal{R}^m(u_0),\mathcal{R}^{m+1}(u_0))
&\leq \lambda_1d(\mathcal{R}^{m-1}(u_0),\mathcal{R}^{m}(u_0))\\
& +\frac{\lambda_2+\lambda_3}{2}[\rho(\mathcal{R}^{m}(u_0),\mathcal{R}^{m-1}(u_0))+\rho(\mathcal{R}^{m+1}(u_0),\mathcal{R}^{m}(u_0))]\\
& +\frac{\lambda_4+\lambda_5}{2}\rho(\mathcal{R}^{m-1}(u_0),\mathcal{R}^{m+1}(u_0)).
\end{align*}
\noindent
Therefore, we get
\begin{equation}\label{eq:3}
\rho(\mathcal{R}^m(u_0),\mathcal{R}^{m+1}(u_0))\leq {\dfrac{\lambda_1+\frac{\lambda_2+\lambda_3}{2}+\frac{\lambda_4+\lambda_5}{2}}{1-\frac{\lambda_2+\lambda_3}{2}-\frac{\lambda_4+\lambda_5}{2}}} \rho(\mathcal{R}^{m}(u_0),\mathcal{R}^{m-1}(u_0)).
\end{equation}
\noindent
Since, $\lambda<1,$ $\dfrac{\lambda_1+\frac{\lambda_2+\lambda_3}{2}+\frac{\lambda_4+\lambda_5}{2}}{1-\frac{\lambda_2+\lambda_3}{2}-\frac{\lambda_4+\lambda_5}{2}} <\lambda.$
%
%
\noindent
Therefore, from equation (\ref{eq:3}), we obtain
$$\rho(\mathcal{R}^m(u_0),\mathcal{R}^{m+1}(u_0))\leq \lambda \rho(\mathcal{R}^{m}(u_0),\mathcal{R}^{m-1}(u_0)).$$
Thus, $$\rho(\mathcal{R}^m(u_0),\mathcal{R}^{m+1}(u_0))\leq \lambda^{m} \rho(\mathcal{R}(u_0),u_0).$$
Let $j\in \mathbb{N},$ then using triangle inequality on an orbit of $u_0$, we have
\begin{align*}
\rho(\mathcal{R}^{m}(u_0),\mathcal{R}^{m+j}(u_0))
& \leq  \rho(\mathcal{R}^{m}(u_0),\mathcal{R}^{m+1}(u_0))+\cdots +\rho(\mathcal{R}^{m+j-1}(u_0),\mathcal{R}^{m+j}(u_0))\\
& \leq \lambda^m \rho(\mathcal{R}(u_0),u_0)+\cdots +\lambda^{m+j-1}\rho(\mathcal{R}(u_0),u_0)\\
& \leq \lambda^m \rho(\mathcal{R}(u_0),u_0)[1+\lambda+\lambda^2+\cdots+ \lambda^{j-1}].
\end{align*}
Taking limit on $n$, we obtain $\{\mathcal{R}^m(u_0)\}$ is $\rho$-Cauchy. Therefore, using $\rho$-completeness of $X$, orbital continuity of $\mathcal{R}$ and Remark \ref{ul}, there exist a point $u_0^*\in X$ satisfying $\mathcal{R}(u_0^*)=u_0^*.$ It is easy to compute the uniqueness.


\noindent
Also, using triangle inequality on an orbit of $u_0$, we get $$\rho(u_0,u_0^*)\leq \dfrac{(2+\lambda)\delta}{2(1-\lambda)}.$$
\end{proof}

\begin{remark}
In the article \cite{C3}, \'Ciri\'c mentioned that a generalized contraction mapping implies a quasi-contraction mapping, but the converse does not hold. Therefore, existence of fixed point of a quasi-contraction mapping in a non-triangular metric space [Theorem 4.1, \cite{KK1}], follows from Theorem \ref{tug}.
\end{remark}

{\Large{\section{\textbf{Application to Hyre-Ulam Stability}}}}

In this section, motivated by Baker's\cite{JB1} result (Theorem \ref{TB1}), we employ the fixed point results deduced in previous section to develop some results on the stability of equation (\ref{eq0}). We start with a lemma to move forward.

\medskip

\begin{lemma}\label{ll}
Let $S$ be a non-empty set and $(X,\rho)$  be a \textit{non-triangular metric space}. Define a mapping $\rho':Y\times Y\rightarrow [0,\infty)$ by 
\begin{equation}\label{rho}
\rho'(g,h)=\sup_{s\in S}\rho(g(s),h(s)).
\end{equation}
Then $\rho'$ is a non-triangular metric. Moreover,  $\rho$-completeness $\implies$ $\rho'$-completeness.
\end{lemma}

\begin{proof}

Clearly, (N1) and (N2) are trivial. Let $h,g\in Y$ and $\{h_m\}$ be a sequence in $Y$ with $$\lim_{m\rightarrow \infty} \rho'(h_m,h)=0 \mbox{ and } \lim_{m\rightarrow \infty} \rho'(h_m,g)=0.$$ Using definition (\ref{rho}), we get $$\lim_{m\rightarrow \infty} \sup_{s\in S}\rho(h_m(s),h(s))=0 \mbox{ and } \lim_{m\rightarrow \infty} \sup_{s\in S}\rho(h_m(s),g(s))=0.$$
Since, $$0=\lim_{m\rightarrow \infty} \sup_{s\in S}\rho(h_m(s),h(s))\geq \sup_{s\in S} \lim_{m\rightarrow \infty} \rho(h_m(s),h(s)).$$ It implies $$\lim_{m\rightarrow \infty} \rho(h_m(s),h(s))=0, ~ \forall  s\in S.$$ Similarly, $$\lim_{m\rightarrow \infty} \rho(h_m(s),g(s))=0, ~ \forall s\in S.$$ Then by (N3), we have $h(s)=g(s),~  \forall s\in S,$ i.e.,  $h=g$. It shows that (N3) holds for $(Y,\rho').$ Therefore, $\rho'$ is a non-triangular metric.

\medskip

At first we show $\rho'$-Cauchy $\implies \rho$-Cauchy.
Take a $\rho'$-Cauchy sequence $\{h_m\}\in Y.$ Then $$\lim_{m\rightarrow \infty}\sup\{\rho'(h_m,h_k):k\geq m\}=0.$$
Therefore, $$\lim_{m\rightarrow \infty} \sup \{\sup_{s\in S}\rho(h_m(s),h_k(s)):k\geq m\}=0.$$ 
This implies $$\lim_{m\rightarrow \infty}\sup \{\rho(h_m(s),h_k(s)): k\geq m\}=0,~  \forall s\in S.$$ Thus, for each $s\in S$, $\{h_m(s)\}$ is $\rho$-Cauchy in $X$. Again $X$ is $\rho$-complete. Consequently, for each $s\in S$, there exists $h(s)\in X$ such that $$\lim_{m\rightarrow \infty} \rho(h_m(s),h(s))=0.$$ Thus $h\in Y.$ We claim that $\displaystyle{\lim_{m\rightarrow \infty}} \rho'(h_m,h)=0.$ Since for any arbitrary $M\in \mathbb{N}$ and for $s\in S$, there exists $P\in \mathbb{N}$ such that $$\rho(h_m(s),h(s))<\frac{1}{M}, ~ \forall  m\geq P.$$ Therefore, for some $Q\in \mathbb{N}$,  $$\sup_{s\in S}\{\rho(h_m(s),h(s))\} \leq\frac{1}{M}, ~ \forall  m\geq Q.$$ Hence, $$\lim_{m\rightarrow \infty}\rho'(h_m,h)=0.$$ This concludes the theorem.
\end{proof}
We now define $\mathcal{O}(h)\in X^S$ by, $$\mathcal{O}(h)(s)=G(s,h(\psi(s))),$$ for all $h\in X^S$, where $G$ is referred in (\ref{eq0}). Thus, $\mathcal{O}$ is an operator on $X^S$. 

We are in a state to establish the stability of (\ref{eq0}) as a utilization of Theorem \ref{tub}.

\begin{theorem}\label{tubu}
Let $(X,\rho)$ be a complete non-triangular metric space, and $S$ be a non-empty set. Suppose $\psi \in S^S$ and $G\in X^{S\times X}$ with 
\begin{enumerate}
\item $\displaystyle{\sup_{s\in S,m\in \mathbb{N}}} \{\rho(g(s),\mathcal{O}^m(g(s)))\}\leq \delta,$ for some $g\in X^S$ and some $\delta>0$,
\item $\displaystyle{\sup_{s\in S}}  \{\rho(G(s,u_1),G(s,u_2))\}\leq \lambda \rho(u_1,u_2),$ for $\lambda \in [0,1)$ and $u_1,u_2\in X.$
\end{enumerate}
Then there is a unique function $h\in X^S$ satisfying $$h(s)=G(s,h(\psi(s))),~ \forall s\in S.$$
\end{theorem}

\begin{proof}
Definition (\ref{rho}) and condition (1) implies that $$\sup_{m\in \mathbb{N}} {\rho'(g,\mathcal{O}^m(g))}\leq \delta.$$ 
Also, from condition (2) we get
\begin{align*}
\rho(\mathcal{O}(g_1)(s),\mathcal{O}(g_2)(s))
& = \rho(G(s,g_1(\psi(s))),G(s,g_2(\psi(s))))\\
& \leq \lambda \rho(g_1(\psi(s)),g_2(\psi(s)))\\
& \leq \lambda\rho'(g_1,g_2), 
\end{align*} 
for all $g_1,g_2\in Y.$
Therefore, $$\rho'(\mathcal{O}(g_1),\mathcal{O}(g_2)) \leq \lambda\rho'(g_1,g_2),\quad \forall g_1,g_2\in Y.$$
Then Theorem \ref{tub} along with the Lemma \ref{ll} assure us, existence of a unique function $h\in Y$ that satisfies $\mathcal{O}(h)=h$.
\end{proof}

\begin{corollary}\label{tubur}
Let $(X,\rho)$ be a complete non-triangular metric space, and $S$ be a non-empty set. Suppose $\psi \in S^S$ and $G\in X^{S\times X}$ with 
\begin{enumerate}
\item $\displaystyle{\sup_{s\in S}} \{\rho(g(s),G(s,g(\psi(s))))\}\leq \delta,$ for some $g\in X^S$ and some $\delta>0$,
\item $\displaystyle{\sup_{s\in S}}  \{\rho(G(s,u_1),G(s,u_2))\}\leq \lambda \rho(u_1,u_2),$ for $\lambda \in [0,1)$ and $u_1,u_2\in X.$
\end{enumerate}
If for some $g\in X^S$, the triangle inequality hold on an orbit of $g$, then there is a unique function $h\in X^S$ satisfying $$h(s)=G(s,h(\psi(s))),~ \forall s\in S.$$ Also, $$\rho(g(s),h(s))\leq \dfrac{\delta}{1-\lambda}.$$ 
\end{corollary}

\begin{proof}
Theorem \ref{tubu} and  \ref{tub1} immediately leads to the result.
\end{proof}

Afterwards, we apply Theorem \ref{tuk} to enquire the stability of (\ref{eq0}). Here we use the notion of orbital continuity \ref{oc} of the operator $\mathcal{O}.$ Let $\mathcal{O}(g)\in X^S,$ then $\mathcal{O}(g,\infty)=\{g,\mathcal{O}(g),\mathcal{O}^2(g),\cdots\}$ is an orbit of $g$. Here $\mathcal{O}^k(g)$'s are as follow:
\begin{align*}
\mathcal{O}(g)(s)& = G(s,g(\psi(s))),\\
\mathcal{O}^2(g)(s)& = G(s,G(\psi(s),g(\psi^2(s)))),\\
& \vdots\\
\mathcal{O}^k(g)(s)& = G(s,G(\psi(s),G(\psi^2(s),G(\psi^3(s)),G(\cdots,G(\psi^{k-1}(s),g(\psi^k(s)))))))).
\end{align*}

\begin{theorem}\label{tuku}
Let $S$ be a non-empty set and $(X,\rho)$ be a complete \textit{non-triangular metric space}, and $\psi \in X^S$, $G\in X^{S\times X}$ given. Suppose that 
\begin{enumerate}
\item $\displaystyle{\sup_{s\in S}} \{\rho(g(s),G(s,g(\psi(s))))\}\leq \delta,$ for some $g\in X^S$ and $\delta>0$,
\item $\displaystyle{\sup_{s\in S}} \{\rho(G(s,u_1),G(s,u_2))\}\leq \lambda [\rho(G(s,u_1),u_1)+ \rho(G(s,u_2),u_2)],$ $\forall u_1,u_2\in X,$ and $\lambda \in [0,\frac{1}{2})$,
\item $\mathcal{O}$ is orbitally continuous.
\end{enumerate}
Then there is a unique $h\in X^S$ satisfying $$h(s)=G(s,h(\psi(s))), ~ \forall  s\in S.$$
\end{theorem}

\begin{proof}
Lemma \ref{ll} provides $(Y,\rho' )$ as a complete non-triangular metric space. Then, condition (1) implies $${\rho'(g,\mathcal{O}(g))}\leq \delta.$$ 
Again by condition (2), we get
\begin{align*}
\rho(\mathcal{O}(g_1)(s),\mathcal{O}(g_2)(s))
& = \rho(G(s,g_1(\psi(s))),G(s,g_2(\psi(s))))\\
& \leq \lambda [\rho(G(s,g_1(\psi(s)),g_1(\psi(s))+ \rho(G(s,g_2(\psi(s)),g_2(\psi(s))] \\
& \leq \lambda [\rho'(\mathcal{O}(g_1),g_1)+\rho'(\mathcal{O}(g_2),g_2)],
\end{align*} 
for all $g_1,g_2\in Y.$
Thus, $$\rho'(\mathcal{O}(g_1),\mathcal{O}(g_2)) \leq \lambda[\rho'(\mathcal{O}(g_1),g_1)+\rho'(\mathcal{O}(g_2),g_2)], ~ \forall g_1,g_2\in Y.$$
Then by using Theorem \ref{tuk}, we have a unique function $h\in Y$ satisfying $\mathcal{O}(h)=h$.
\end{proof}

\begin{remark}
Suppose, the triangle inequality holds on an orbit of $g$, for some $g\in Y.$ Then, $$\rho'(g,h)\leq \dfrac{(1+\lambda)\delta}{1-2\lambda}.$$
\end{remark}

For the following stability result we need a weaker version of triangle inequality, i.e., triangle inequality on an orbit $O(g,\infty)$.

\begin{theorem}\label{tucg}
Let $S$ be a non-empty set and $(X,\rho)$ be a complete \textit{non-triangular metric space}. Suppose, $\psi \in X^S$, $G\in X^{S\times X}$ be such that
\begin{enumerate}
\item For some $g\in X^S$, triangle inequality holds on the set $\{g, \mathcal{O}(g), \mathcal{O}^2(g),\cdots \}$ and $\displaystyle{\sup_{s\in S}} \{\rho(g(s),G(s,g(\psi(s))))\}\leq \delta,\quad \mbox{for some }\delta>0,$
\item $\displaystyle{\sup_{s\in S}} \{\rho(G(s,u_1),G(s,u_2))\}\leq \lambda [\rho(G(s,u_1),u_2)+ \rho(G(s,u_2),u_1)],~\forall u_1, u_2\in X,$ and  $\lambda \in [0,\frac{1}{2}),$
\item $\mathcal{O}$ is orbitally continuous.
\end{enumerate}
Then there exists a unique $h\in X^S$ such that $$h(s)=G(s,h(\psi(s))), ~ \forall s\in S.$$ Moreover, for all  $s\in S$, $$\rho(g(s),h(s))\leq \dfrac{(1+\lambda)\delta}{1-2\lambda}.$$
\end{theorem}

\begin{proof}
Lemma \ref{ll} provides $(Y,\rho' )$ as a complete non-triangular metric space. Thus, condition (1) implies  $$\rho'(g,\mathcal{O}(g))\leq \delta.$$ Also, by condition (2) we get
\begin{align*}
\rho(\mathcal{O}(g_1)(s),\mathcal{O}(g_2)(s))
& = \rho(G(s,g_1(\psi(s))),G(s,g_2(\psi(s))))\\
& \leq \lambda [\rho(G(s,g_1(\psi(s)),g_2(\psi(s))+ \rho(G(s,g_2(\psi(s)),g_1(\psi(s))] \\
& \leq \lambda[\rho'(\mathcal{O}(g_1),g_2)+\rho'(\mathcal{O}(g_2),g_1)],
\end{align*} 
for all $g_1,g_2\in Y.$
Then $$\rho'(\mathcal{O}(g_1),\mathcal{O}(g_2)) \leq \lambda[\rho'(\mathcal{O}(g_1),g_2)+\rho'(\mathcal{O}(g_2),g_1)], ~\forall g_1,g_2\in Y.$$
Then Theorem \ref{tuc} shows that, there is a unique function $h\in Y$ satisfying $\mathcal{O}(h)=h$. Also, triangle inequality on an orbit of $g$ ensures that $$\rho'(g,h)\leq \dfrac{(1+\lambda)\delta}{1-2\lambda}.$$
\end{proof}

Now we apply Theorem \ref{tug} to develop another stability result of equation (\ref{eq0}). It is a kind of generalization using an additional condition of all the theorems described in this section.

\begin{theorem}\label{tucgu}
Let $S$ be a non-empty set and $(X,\rho)$ be a complete \textit{non-triangular metric space}. Suppose, $\psi \in X^S$, $G\in X^{S\times X}$ be such that 
\begin{enumerate}
\item For some $g\in X^S$, triangle inequality holds on the set $\{g,\mathcal{O}(g),\mathcal{O}^2(g),\cdots \}$ and $\displaystyle{\sup_{s\in S}} \{\rho(g(s),G(s,g(\psi(s))))\}\leq \delta, \mbox{ for some }\delta>0,$
\item For all $u_1,u_2\in X$ and $\lambda\in [0,1)$,
\begin{align*}
~\qquad\sup_{s\in S} {\rho(G(s,u_1),G(s,u_2))}&\leq \lambda_1\rho(u_1,u_2)+\lambda_2\rho(G(s,u_1),u_1)+\lambda_3 \rho(G(s,u_2),u_2)\\
&+\lambda_4\rho(G(s,u_1),u_2)+\lambda_5\rho(G(s,u_2),u_1),
\end{align*}
where $\lambda_i:Y\times Y\rightarrow \mathbb{R}_+$, $i\in\{1, \dots, 5\}$, with $\sum_{i=1}^{5}\lambda_{i}(u_1,u_2)\leq \lambda$,
\item $\mathcal{O}$ is orbitally continuous.
\end{enumerate}
Then there is a unique function $h\in X^S$ such that $$h(s)=G(s,h(\psi(s))) \mbox{ and } \rho(g(s),h(s))\leq \dfrac{(2+\lambda)\delta}{2(1-\lambda)}, ~ \forall  s\in S.$$
\end{theorem}

\begin{proof}
By using Theorem \ref{tug}, one can easily establish the result.
\end{proof}

{\Large{\section{\textbf{Consequence of Stability Theory}}}}

In the article \cite{JB1}, Baker investigated Hyre-Ulam stability for the equation:
\begin{equation}\label{eq:5}
f(s)=\lambda(s)f(\psi(s))+ g(s). 
\end{equation}
The author developed this result in Banach spaces: 

\begin{theorem}\cite{JB1}
Let $X$ be a  Banach space (real or complex) and $S$ be a non-empty set. Consider $\psi:S\rightarrow S,~ g:S\rightarrow X,~\lambda:S\rightarrow \mathbb{R} ~(\mbox{or }  \mathbb{C})$ with $$\lvert{\lambda(s)\rvert}\leq \lambda, \quad \quad s\in S, $$ for some $0\leq \lambda <1.$ Assume that $f:S\rightarrow X$ satisfies $$\lVert{f(s)-\lambda(s)f(\psi(s))- g(s)\rVert}\leq \delta, \quad \quad s\in S,$$ where $\delta>0$ is a constant. Then there exists a unique function $h:S\rightarrow X$ satisfying equation (\ref{eq:5}) and $$\lVert{f(s)-h(s)\rVert}\leq \frac{\delta}{1-\lambda}, \quad \quad s\in S.$$
\end{theorem} 

In this section, we deduce a general version of above theorem in topological vector spaces, as an implication of Corollary \ref{tubur}.

\begin{theorem}
Let $S$ be a non-empty set and $V$ be a topological vector space over $\mathbb{R}\mbox{ or }  \mathbb{C}$. Suppose $\rho:V\times V\rightarrow \mathbb{R}_{+}$ be a function such that $(V,\rho)$ is a complete non-triangular metric space. Let $p:V\rightarrow \mathbb{R} ~\mbox{or }  \mathbb{C}$ be a sublinear functional such that $\lvert{p(v_1-v_2)\rvert}=\rho(v_1,v_2),~ \forall v_1,v_2\in V$. Consider $\psi:S\rightarrow S,~ B:S\rightarrow V,~\lambda:S\rightarrow \mathbb{R}_{+}$ with $$\lambda(s)\leq \lambda, \quad \quad s\in S, $$ for some $0\leq \lambda <1.$ Suppose for some $\delta>0$, $A:S\rightarrow V$ satisfies
\begin{equation}
\rho(A(s),\lambda(s)A(\psi(s))+B(s))\leq \delta, \quad \quad s\in S.\label{eq:6}
\end{equation} 
If the triangle inequality holds on an orbit of $A$, then a unique solution $\Lambda:S\rightarrow V$ exists, satisfying $$\Lambda(s)=\lambda(s)\Lambda(\psi(s))+ B(s).$$ Moreover, $$\rho(A(s),\Lambda(s))\leq \frac{\delta}{1-\lambda}, \quad \quad s\in S.$$
\end{theorem} 

\begin{proof}
Let $G:S\times V\rightarrow V$ be such that $$G(s,v)=\lambda(s)v+ B(s), \quad v\in V.$$ Then, $$G(s,A(\psi(s)))=\lambda(s)A(\psi(s))+ B(s).$$ Clearly, by the hypothesis (\ref{eq:6}), condition (1) of Corollary \ref{tubur} is satisfied. Further, \begin{align*}
\rho(G(s,A(\psi(s)), G(s,A'(\psi(s)))
& =\lvert{p(\lambda(s)A(\psi(s))+ B(s)-\lambda(s)A'(\psi(s))-B(s))\rvert}\\
& \leq \lambda(s)\lvert{p(A(\psi(s))-A'(\psi(s)))\rvert} \\
& \leq\lambda \rho(A(\psi(s)),A'(\psi(s))).
\end{align*}
for all $A, A' \in V^S.$
Evidently, condition (2) of Corollary \ref{tubur} holds. Therefore, applying Corollary \ref{tubur}, we obtain, a unique function $\Lambda\in V^S$ satisfying the equation $$\Lambda(s)=\lambda(s)\Lambda(\psi(s))+ B(s),\text{ and }\rho(A(s), \Lambda(s))\leq \frac{\delta}{1-\lambda}. $$
\end{proof}

\begin{center}
\textbf{\large {Conclusion}}
\end{center}
Metric fixed point theory has a crucial role in solving various real-world problems. In some cases, for example, optimization problems, machine learning, robotics and navigations, and computer network problems e.g., routing decisions and inefficient network traffic distributions; the triangle inequality may not properly hold. In view of these practical scenarios, it is important to develop the fixed point results in the structure of non-triangular metric spaces.

In this article, we establish various well-known fixed point results in non-triangular metric spaces. Furthermore, we investigate weather the solutions obtained in these problems are stable or not. To show the stability, we develop Hyre-Ulam stability of an equation using the deduced fixed point results in the setting of non-triangular metric spaces. Moreover, we measure the quantity that how much an approximate solution can differ the most from the solution.

\begin{Acknowledgement} The authors are  thankful to Subhadip Pal, NIT Durgapur,
for his  suggestions during the preparation of the manuscript.
\end{Acknowledgement}

\bibliographystyle{plain}

\end{document}